\documentclass[12pt]{article}
\usepackage{amsmath}
\usepackage{amssymb}
\usepackage{mathrsfs}
\usepackage{amssymb}
\usepackage{latexsym}
\usepackage{amsthm}
\usepackage{mathrsfs}
\usepackage{enumitem}
\usepackage[colorlinks,
            linkcolor=red,
            anchorcolor=blue,
            citecolor=blue
            ]{hyperref}
\newtheorem{thm}{Theorem}[section]
\newtheorem{lem}[thm]{Lemma}

\newtheorem{conj}[thm]{Conjecture}

\newcommand{\floor}[1]{\ensuremath{\left\lfloor #1 \right\rfloor}} 

\begin{document}

\begin{center}
{\large \bf  The Local $h$-Polynomials of Cluster Subdivisions\\
	 Have Only Real Zeros}
\end{center}

\begin{center}
Philip B. Zhang\\[6pt]

College of Mathematical Science \\
Tianjin Normal University, Tianjin  300387, P. R. China\\[8pt]

Email: {\tt zhangbiaonk@163.com}
\end{center}

\noindent\textbf{Abstract.} 
Athanasiadis raised the question whether the local $h$-polynomials of type $A$ cluster subdivisions have only real zeros. In this paper, we confirm this conjecture and prove the real-rootedness of local $h$-polynomials for all the other Cartan--Killing types. Our proofs mainly involve multiplier sequences and Chebyshev polynomials of the second kind.

\noindent \emph{AMS Classification 2010:}  05A15, 05E45,  26C10, 52B45


\section{Introduction}
In this paper, we confirm a question of Athanasiadis that the  local $h$-polynomials of type $A$ cluster subdivisions have only real zeros. We also show the real-rootedness of local $h$-polynomials for all the other Cartan--Killing types.
%

We first give an overview of local $h$-polynomials. 
The notion of local $h$-polynomials was introduced by Stanley \cite{Stanley1992Subdivisions} in his  study of the face enumeration of subdivisions of complexes.
Let $V$ be an $n$-element vertex set.  Given a  simplicial  subdivision $\Gamma$ of the abstract
simplex $2^V$, the local $h$-polynomial $\ell_V (\Gamma, x)$ is defined as an alternating
sum of the $h$-polynomials of the restrictions of $\Gamma$ to the
faces of $2^V$, namely,		
\begin{align*}
 \ell_{V} \left(\Gamma, x \right) = \sum_{F \subseteq  V} (-1)^{n-|F|} \, h\left(\Gamma_{F},x\right),
\end{align*}
where $h\left(\Gamma_{F},x\right)$ is the $h$-polynomial of $\Gamma_{F}$.
Stanley \cite{Stanley1992Subdivisions} also showed that $\ell_{V} \left(\Gamma, x \right)$ has nonnegative and symmetric coefficients, hence
the local $h$-polynomial  can be expressed as 
  \begin{equation} \label{eq:defclocalg}
    \ell_{V} (\Gamma, x) \ = \ \sum_{i=0}^{\lfloor n/2 \rfloor} \, \xi_i
    \, x^i (1+x)^{n-2i}.
  \end{equation}
Athanasiadis \cite{Athanasiadis2017survey} made an excellent survey on this topic and raised an open question whether local $h$-polynomials for several families of subdivisions have only real zeros. 

This paper is concerned with cluster subdivisions.
Let $I$ be an $n$-element set and $\Phi = \{a_i:i\in I\}$ be a root system. 
The cluster complex $\Delta(\Phi)$, studied by Fomin and Zelevinsky \cite{Fomin2002Cluster, Fomin2003$Y$}, is a simplicial complex on the vertex set of positive roots and negative simple roots.
The positive cluster complex $\Delta_+(\Phi)$ is the restriction of $\Delta(\Phi)$ on the positive roots.
It naturally defines a geometric subdivision of the simplex on the vertex set of simple roots of $\Phi$,
the so-called \emph{cluster subdivision $\Gamma(\Phi)$}.
The local $h$-polynomial $\ell_I(\Gamma(\Phi),x)$ is given by 
$$\ell_I(\Gamma(\Phi),x) = \sum_{J\subseteq I} (-1)^{|I\backslash J|}h(\Delta_+(\Phi_{J}),x),$$
where $\Phi_{J}$ is the parabolic root subsystem of $\Phi$ with respect to $J$.

Although closed form expressions for the local $h$-polynomials of type A and type B were not found until now, the following result of Athanasiadis and Savvidou \cite{Athanasiadis2011/12local} gave explicit expressions of the numbers $\xi_i$ defined by \eqref{eq:defclocalg}.

\begin{lem} [{\cite[Theorem 1.2]{Athanasiadis2011/12local}}]\label{thm:clocalg}
Let $\Phi$ be an irreducible root system of rank $n$ and Cartan--Killing type $\mathcal X$
and let $\xi_i (\Phi)$ be the integers uniquely defined by \eqref{eq:defclocalg}.
Then $\xi_0 (\Phi) = 0$ and
  $$ \xi_i (\Phi) \ = \ \begin{cases}
    \displaystyle \frac{1}{n-i+1} \binom{n}{i} \binom{n-i-1}{i-1},
    & \text{if \ $\mathcal X = A_n$} \\ & \\
    \displaystyle \binom{n}{i} \binom{n-i-1}{i-1}, & \text{if \ $\mathcal X = B_n$} \\
    & \\ \displaystyle \frac{n-2}{i} \binom{2i-2}{i-1} \binom{n-2}{2i-2},
    & \text{if \ $\mathcal X = D_n$}  \end{cases} $$

\bigskip
\noindent
for $1 \le i \le \lfloor n/2 \rfloor$. Moreover,

 $$ \sum_{i=0}^{\lfloor n/2 \rfloor} \, \xi_i (\Phi) x^i \ = \ \begin{cases}
    (m-2)x, & \text{if \ $\mathcal X = I_2(m)$} \\
    8x, & \text{if \ $\mathcal X = H_3$} \\
    42x + 40x^2, & \text{if \ $\mathcal X = H_4$} \\
    10x + 9x^2, & \text{if \ $\mathcal X = F_4$} \\
    7x + 35x^2+ 13x^3, & \text{if \ $\mathcal X = E_6$} \\
    16x + 124x^2 + 112x^3, & \text{if \ $\mathcal X = E_7$} \\
    44x + 484x^2 + 784x^3 + 120x^4, & \text{if \ $\mathcal X = E_8$}.  \end{cases} $$
\end{lem}

Athanasiadis \cite{Athanasiadis2017survey} made the following conjecture.

\begin{conj}
 The local $h$-polynomial of type $A$ cluster subdivision of the simplex  has only real zeros.
\end{conj}

In this paper, we confirm this conjecture  and furthermore prove the real-rootedness for all the other types.

\begin{thm}\label{thm:main}
 For any irreducible root system, the local $h$-polynomial of the cluster subdivision of the simplex  has only real zeros.
\end{thm}


The remainder of this paper is organized as follows.
In Section \ref{sect:pre}, we give an overview of the theory of multiplier sequences.
In Section \ref{sect:proof}, we present our proof of Theorem \ref{thm:main}.

\section{Preliminaries}\label{sect:pre}
In this section, we present some background on real-rooted polynomials. Several basic facts about multiplier sequences are given.
Recall that a  sequence of real numbers $\{\lambda_{k}\}_{k=0}^{\infty}$ is a
multiplier sequence, if for every polynomial $\sum_{k=0}^{n}a_{k}z^{k}$
with all zeros real, the polynomial $\sum_{k=0}^{n}\lambda_{k}a_{k}z^{k}$
is either identically zero or has only real zeros.
In the following, we shall list some related facts which will be used in the paper. 
For a complete introduction of multiplier sequences, we refer the reader to \cite{Craven1996Problems, Craven2004Composition, Rahman2002Analytic}.

A fundamental theorem of multiplier sequences is due to P{\'o}lya and Schur \cite{Polya1914Uber}. Before introducing their result, we recall the notion of Laguerre--P{\'o}lya class.
An entire function $\phi(x)=\sum_{i=0}^{\infty}\gamma_k \frac{x^k}{k!}$ is in the \emph{Laguerre--P{\'o}lya class}, written 
$\phi \in \mathscr{L}\mbox{-}\mathscr{P}$
if it can be written as 
\begin{align*}
\phi(x) = c x^m e^{-ax^2+bx}\prod_{k=1}^{\omega}(1+\frac{x}{x_k})e^{-\frac{x}{x_k}}, \quad (0\le \omega \le \infty), 
\end{align*}
where $b,c,x_k\in \mathbb{R}$, $m$ is a non-negative integer, $a\ge 0$, $x_k\neq 0$ and $\sum_{k=1}^{\omega} \frac{1}{x_k^2}<\infty$.
Let $\mathscr{L}\mbox{-}\mathscr{P}^+$ denote the set of functions in the Laguerre--P{\'o}lya class with nonnegative coefficients, and  $\mathscr{L}\mbox{-}\mathscr{P}(-\infty,0]$ denote the set of functions in the Laguerre--P{\'o}lya class that have only non-positive zeros. 
A remarkable property is that an entire function is in the Laguerre--P{\'o}lya class if and only if it is a locally uniform limit of real polynomials which have only real zeros. 

A complete characterization of multiplier sequences was given by P{\'o}lya
and Schur \cite{Polya1914Uber}. 
\begin{thm}[P{\'o}lya--Schur]\label{thm:PS}
Let $\{\lambda_k\}_{k=0}^{\infty}$ be a sequence of real numbers. The following statements are equivalent:
\begin{enumerate}
 \item[(i)] $\{\lambda_k\}_{k=0}^{\infty}$ is a multiplier sequence;
 \item[(ii)] For any non-negative integer $n$, either the polynomial $\sum_{k=0}^{n}{\binom{n}{k}}\lambda_k x^k$ has only real zeros of the same sign or it is identically zero.
 \item[(iii)] Either $\sum_{k=0}^{\infty} \lambda_k \frac{x^k}{k!}$ or $\sum_{k=0}^{\infty} (-1)^k \lambda_k \frac{x^k}{k!}$ belongs to $\mathscr{L}\mbox{-}\mathscr{P}^+$.
\end{enumerate}
\end{thm}

For convenience, we let $\frac{1}{k!}$ be zero whenever $k$ is a negative integer.
By Theorem \ref{thm:PS}, we obtain the following result. 
\begin{lem}\label{lem:1}
 For any postive integer $n$, the sequence $\{ \frac{1}{(n-k)!}\}_{k=0}^{\infty}$ is a multiplier sequence.
\end{lem}
\begin{proof}
 Clearly, the function $$\sum_{k=0}^{\infty}\frac{1}{(n-k)!}\frac{x^k}{k!} = \frac{1}{n!}(1+x)^n$$ has only real zeros.
 This completes the proof by Theorem \ref{thm:PS}.
\end{proof}

The following result of Laguerre can produce several multiplier sequences. 
\begin{thm}[{\cite{Craven2004Composition}}]\label{thm:Laguerre}
 If $\phi(x) \in \mathscr{L}\mbox{-}\mathscr{P} (-\infty,0]$, then $\{\phi(k)\}_{k=0}^{\infty}$ is a multiplier sequence.
\end{thm}
The following identity of the gamma function, due to Weierstrass,
\begin{align*}
 {\Gamma(x)} = \frac{1}{x} \exp(-\gamma x) \prod_{n=1}^{\infty}{(1+\frac{x}{n})}^{-1} \exp(\frac{x}{n}),
\end{align*} 
where $\gamma \approx 0.577216\cdots$ is the Euler--Mascheroni constant, shows that 
$\frac{1}{\Gamma(x)}$
belongs to $\mathscr{L}\mbox{-}\mathscr{P} (-\infty,0]$. Hence, it follows from Theorem \ref{thm:Laguerre} that 
\begin{lem}\label{lem:2}
The sequence $\{\frac{1}{k!}\}_{k=0}^{\infty}
=\{\frac{1}{\Gamma(k+1)}\}_{k=0}^{\infty}$ is a multiplier sequence. 
\end{lem}
We now give two multiplier sequences, which will be used in the next section.
\begin{lem}\label{lem:ms}
 For any positive integer $n$, 
 the sequence $\{\frac{1}{i! (n-i)!}\}_{i\ge 0}$ is a multiplier sequence.  
\end{lem}
\begin{proof}
By the definition, the Hadamard product (termwise product) of two multiplier sequences is also a multiplier sequence.
Hence, by Lemma \ref{lem:1} and Lemma \ref{lem:2}, it follows that $\{\frac{1}{i! (n-i)!}\}_{i\ge 0}$ is a multiplier sequence.  This completes the proof.
\end{proof}

%

Before ending this section, we address the following elementary but useful fact .

\begin{lem}[{\cite[ Observation 4.2]{Petersen2015Eulerian}}]\label{lem:transformation}
    
  If a polynomial $\ell (x)$ has symmetric coefficients, then 
  \begin{align*}
  \ell (x) = \sum_{i=1}^{\lfloor n/2 \rfloor} \, \xi_i  \, x^i (1+x)^{n-2i}
  \end{align*} has only negative real zeros if and only if so does 
the polynomial 
  \begin{align*}
  \xi (x) = \sum_{i=1}^{\lfloor n/2 \rfloor} \, \xi_i
     \, x^i.
  \end{align*}
\end{lem}
%
%
%

\section{Real-rootedness of local  \texorpdfstring{$h$}{Lg}-polynomials of cluster subdivisions}\label{sect:proof}
In this section, we shall give our proof of Theorem \ref{thm:main} case by case.
Combining Lemma \ref{thm:clocalg} and Lemma \ref{lem:transformation}, one can easily check the local $h$-polynomials for the exceptional groups have only real zeros.
In the following, we shall discuss the case of type $A$, type $B$ and type $D$, respectively.

\subsection{Type  \texorpdfstring{$A$}{Lg}}
In this subsection we deal with the real-rootedness of 
\begin{align*}
 \ell_I (\Gamma (A_n), x) = \sum_{i=1}^{\floor{n/2}}\frac{1}{n-i+1}\binom{n}{i}\binom{n-i-1}{i-1}x^i(1+x)^{n-2i}.
\end{align*} 
With the aid of Lemma \ref{lem:transformation}, we turn our attention to the following polynomial
\begin{align}\label{eq:A1}
 \xi_I (\Gamma (A_n), x) =
\sum_{i=1}^{\floor{n/2}}\frac{n!}{i!(n-i+1)!}\binom{n-i-1}{i-1}x^i.
\end{align}
The main result of this subsection is as follows.
\begin{thm}\label{thm:A}
 For any positive integer $n$, the polynomial $\xi_I (\Gamma (A_n), x)$ has only real zeros. 
\end{thm}

We first consider the real-rootedness of the following polynomial
\begin{align*}
 \sum_{i=1}^{\floor{n/2}}\binom{n-i-1}{i-1}x^i.
\end{align*} 
Since 
\begin{align*}
 \sum_{i=1}^{\floor{n/2}}\binom{n-i-1}{i-1}x^i =  \sum_{i=1}^{\floor{n/2}}\binom{n-2-(i-1)}{i-1}x^i  = x  \sum_{j=0}^{\floor{n/2}-1}\binom{n-2-j}{j}x^j,
\end{align*}
we focus on the following polynomial 
\begin{align*}
 H_n(x)=\sum_{j=0}^{\floor{n/2}}\binom{n-j}{j}x^j.
\end{align*}
\begin{lem}\label{lem:H}
 For any postive integer $n$, the polynomial $H_n(x)$ has only negative and simple zeros. 
\end{lem}

\begin{proof}
The polynomial $H_n(x)$  is closely related to the Chebyshev polynomial of the second kind,
\begin{align*}
 U_{n}(y)=\sum_{k=0}^{\floor{n/2}}(-1)^k\binom{n-k}{k}(2y)^{n-2k}.
\end{align*}
Replacing $y$ by $1/2y$, we have 
\begin{align}\label{Chebyshev}
 y^n U_{n}(\frac{1}{2y}) = \sum_{k=0}^{\floor{n/2}}\binom{n-k}{k}(-y^2)^{k}.
\end{align}
From its trigonometric definition,  the Chebyshev polynomial of the second kind satisfies
\begin{align*}
 U_n(\cos \theta)=\frac{\sin (n+1)\theta}{\sin \theta}.
\end{align*}
Hence, we get all the zeros of $U_n(y)$, which are $\cos (\frac{k}{n+1}\pi)$, where $k=1,2,\ldots,n$.
Together with \eqref{Chebyshev}, it follows that $-\frac{1}{4}\sec^2(\frac{k}{n+1}\pi)$, where  $k =1, 2, \ldots,  \floor{n/2}$, are the zeros of $H_{n}(x)$.
Therefore, the zeros of $H_n(x)$ are real and simple.
This completes the proof.
\end{proof}

Now we are able to prove Theorem \ref{thm:A}.

\begin{proof}[Proof of Theorem \ref{thm:A}]
By Lemma \ref{lem:H}, the polynomial 
$$\sum_{i=1}^{\floor{n/2}}\binom{n-i-1}{i-1}x^i = x H_{n-2}(x)$$
has only real zeros.
By Lemma \ref{lem:ms}, the sequence $\{\frac{1}{i!(n-i+1)!}\}_{i\ge 0}$ is a multiplier sequence. 
Hence, 
\begin{align*}
 \xi_I (\Gamma (A_n), x) = \sum_{i=1}^{\floor{n/2}}\frac{n!}{i!(n-i+1)!}\binom{n-i-1}{i-1}x^i.
\end{align*} 
has only real zeros. 
This completes the proof.


\end{proof}

\subsection{Type  \texorpdfstring{$B$}{Lg}}
We now consider the real-rootedness of the following polynomials 
\begin{align*}
 \ell_I (\Gamma (B_n), x) = \sum_{i=1}^{\floor{n/2}}\binom{n}{i}\binom{n-i-1}{i-1}x^i(1+x)^{n-2i}
\end{align*} for any postive integer $n$.

Along similar lines of the above section, one can show that 
\begin{thm}\label{thm:B}
For any positive integer $n$, the polynomial $\ell_I (\Gamma (B_n), x) $ has only real zeros.
\end{thm}
\begin{proof}
By Lemma \ref{lem:H}, the polynomial 
$$\sum_{i=1}^{\floor{n/2}}\binom{n-i-1}{i-1}x^i = x H_{n-2}(x)$$
has only real zeros.
By Lemma \ref{lem:ms}, the sequence $\{\frac{1}{i!(n-i)!}\}_{i\ge 0}$ is a multiplier sequence. 
For any positive integer $n$, the polynomial 
\begin{align*}
 \xi_I (\Gamma (B_n), x) =
 \sum_{i=1}^{\floor{n/2}}\binom{n}{i}\binom{n-i-1}{i-1}x^i =
 \ n! \sum_{i=1}^{\floor{n/2}}\frac{1}{i!(n-i)!}\binom{n-i-1}{i-1}x^i
\end{align*}
has only real zeros. 
By Lemma \ref{lem:transformation} we obtain that $\ell_I (\Gamma (B_n), x)$ has only real zeros. 
\end{proof}

\subsection{Type  \texorpdfstring{$D$}{Lg}}
We now consider the real-rootedness of the following polynomials
\begin{align*}
 \ell_I (\Gamma (D_n), x) = \sum_{i=1}^{\floor{n/2}}\frac{n-2}{i}\binom{2i-2}{i-1}\binom{n-2}{2i-2}x^i(1+x)^{n-2i}
\end{align*}
for any postive integer $n\ge2$.

From the following identity \cite{Braenden2011Iterated} of \emph{Narayana polynomials}:
\begin{align*}
 \sum_{i=0}^{\floor{n/2}}\frac{1}{i+1}\binom{2i}{i}\binom{n}{2i}x^i(1+x)^{n-2i}=\sum_{i=0}^{n}\frac{1}{n+1}\binom{n+1}{i}\binom{n+1}{i+1}x^i,
\end{align*}
we get an expression of  $\ell_I (\Gamma (D_n), x)$,
\begin{align*}
 \ell_I (\Gamma (D_n), x) & = \sum_{i=1}^{\floor{n/2}}\frac{n-2}{i}\binom{2i-2}{i-1}\binom{n-2}{2i-2}x^i(1+x)^{n-2i}\\[6pt]
 &= (n-2) x \sum_{i=0}^{\floor{(n-2)/2}}\frac{1}{i+1}\binom{2i}{i}\binom{n-2}{2i}x^i(1+x)^{n-2-2i}\\[6pt]
&= (n-2)x \sum_{i=0}^{n-2}\frac{1}{n-1}\binom{n-1}{i}\binom{n-1}{i+1}x^i.
\end{align*}
The Narayana polynomials have been known to be real-rooted, see \cite{Braenden2011Iterated, Liu2007unified}.
Hence, we are able to derive the following result.
\begin{thm}\label{thm:D}
For any positive integer $n\ge2$, the polynomial $\ell_I (\Gamma (D_n), x) $ has only real zeros.
\end{thm}

\vskip 3mm
\noindent{\bf Acknowledgements.} 
This work was supported by the National Science Foundation of China (Nos. 11626172, 11701424), the TJNU Funding for Scholars Studying Abroad,  the PHD Program of TJNU (No. XB1616) and MECF of Tianjin (No. JW1713).


\end{document}